\documentclass[reqno,11pt]{amsart}
\usepackage{amsmath, latexsym, amsfonts, amssymb, amsthm, amscd}
\usepackage{graphics,epsf,psfrag,epsfig}

\numberwithin{equation}{section}

\newtheorem{theorem}{Theorem}[section]
\newtheorem{lemma}[theorem]{Lemma}
\newtheorem{proposition}[theorem]{Proposition}

\newtheorem{rem}[theorem]{Remark}

\renewcommand{\ge}{\geq}
\renewcommand{\le}{\leq}
\newcommand{\ind}{\mathbf{1}}

\renewcommand{\tilde}{\widetilde}

\newcommand{\cC}{{\ensuremath{\mathcal C}} }

\newcommand{\cT}{{\ensuremath{\mathcal T}} }

\DeclareMathSymbol{\leqslant}{\mathalpha}{AMSa}{"36} % nicer `smaller or equal'
\DeclareMathSymbol{\geqslant}{\mathalpha}{AMSa}{"3E} % nicer `larger or equal'
\DeclareMathSymbol{\eset}{\mathalpha}{AMSb}{"3F}     % nicer `emptyset'
\renewcommand{\leq}{\;\leqslant\;}                   % redef. of < or =
\renewcommand{\geq}{\;\geqslant\;}                   % redef. of > or =
\newcommand{\dd}{\,\text{\rm d}}             % a straight d for differentials
\newcommand{\bbE}{{\ensuremath{\mathbb E}} }

\newcommand{\bbP}{{\ensuremath{\mathbb P}} }

\newcommand{\bbZ}{{\ensuremath{\mathbb Z}} }

\newcommand{\gb}{\beta}
\newcommand{\gga}{\gamma}            % \gg already exists...
\newcommand{\gd}{\delta}
\newcommand{\gep}{\varepsilon}       % \ge already exists...

\newcommand{\go}{\omega}

\makeatletter
\def\captionfont@{\footnotesize}
\def\captionheadfont@{\scshape}

\long\def\@makecaption#1#2{%
  \vspace{2mm}
  \setbox\@tempboxa\vbox{\color@setgroup
    \advance\hsize-6pc\noindent
    \captionfont@\captionheadfont@#1\@xp\@ifnotempty\@xp
        {\@cdr#2\@nil}{.\captionfont@\upshape\enspace#2}%
    \unskip\kern-6pc\par
    \global\setbox\@ne\lastbox\color@endgroup}%
  \ifhbox\@ne % the normal case
    \setbox\@ne\hbox{\unhbox\@ne\unskip\unskip\unpenalty\unkern}%
  \fi
  \ifdim\wd\@tempboxa=\z@ % this means caption will fit on one line
    \setbox\@ne\hbox to\columnwidth{\hss\kern-6pc\box\@ne\hss}%
  \else % tempboxa contained more than one line
    \setbox\@ne\vbox{\unvbox\@tempboxa\parskip\z@skip
        \noindent\unhbox\@ne\advance\hsize-6pc\par}%
\fi
  \ifnum\@tempcnta<64 % if the float IS a figure...
    \addvspace\abovecaptionskip
    \moveright 3pc\box\@ne
  \else % if the float IS NOT a figure...
    \moveright 3pc\box\@ne
    \nobreak
    \vskip\belowcaptionskip
  \fi
\relax
}
\makeatother
\def\writefig#1 #2 #3 {\rlap{\kern #1 truecm
\raise #2 truecm \hbox{#3}}}

\title[Connective constants in high-dimension]
{Existence of a non-averaging regime for the self-avoiding walk on a high-dimensional infinite percolation cluster}

\address{CEREMADE, Place du Mar\'echal De Lattre De Tassigny
75775 PARIS CEDEX 16 - FRANCE}
\email{lacoin@ceremade.dauphine.fr}
\author{Hubert Lacoin}

\begin{document}

\begin{abstract}
Let $Z_N$ be the number of self-avoiding paths of length $N$ starting from the origin on the infinite  cluster obtained after performing Bernoulli percolation  on $\bbZ^d$ with parameter $p>p_c(\bbZ^d)$. The object of this paper is to study the connective constant of the dilute lattice
$\limsup_{N\to \infty} Z_N^{1/N}$, which is a non-random quantity. We want to investigate if  the  inequality 
$\limsup_{N\to \infty} (Z_N)^{1/N} \le \lim_{N\to \infty} \bbE[Z_N]^{1/N}$ obtained with the Borel-Cantelli Lemma is strict or not. In other words, 
we want to know the the quenched and annealed versions of the connective constant are the same. On a heuristic level, this indicates whether or not 
localization of the trajectories occurs.
We prove that when $d$ is sufficiently large there exists $p^{(2)}_c>p_c$ such that the inequality is strict for $p\in (p_c,p^{(2)}_c)$.

\end{abstract}

\maketitle

\section{Introduction}

In this paper we continue a study initiated in \cite{cf:conn} concerning self-avoiding walk on the infinite percolation cluster obtained 
after performing supercritical percolation on $\bbZ^d$. This problem has been studied extensively in the physics litterature, using non-rigorous methods
(see \cite{cf:CK} for the first paper on the subject, and \cite{cf:CR, cf:H2, cf:HM, cf:BKC} for later contributions), but its understanding 
from a mathematical point view remains very weak, the main reason being that self-avoiding walk on $\bbZ^d$ is, in many respects 
already a big challenge for mathematician (see \cite{cf:Slade} for a review).

\medskip

Our starting point is to assume that some of the properties of the walk, such as the replica overlap 
(i.e.\ the mean proportion of edges shared by two independent trajectories) and the end-to-end distance can be deduced from the asymptotic 
behavior of the partition function $Z_N$, which is simply the number of 
open self-avoiding path of length $N$ from the origin. This belief is somehow ascertained by both the physics literature 
(e.g. see \cite{cf:LDM}) and analogies  with rigorous results obtained for directed polymers in a random environment \cite{cf:Com},
but has not yet been brought on rigorous ground (and we will not do it in this paper). Hence our main object of study will be
the growth rate $\limsup_{N\to \infty} (Z_N)^{1/N}$.

\medskip

In \cite{cf:conn} we have aproved the existence of a quenched connective constant for the percolation cluster, or in other words, 
that the upper growth-rate of $Z_N$ is not a random variable. We also proved that in two dimension, this upper-growth rate is always strictly smaller 
than the growth rate of the expected value, underlining a localization phenomenon.
Our aim in this paper is to show that there is a phase where localization occur also in high dimension: 
that when $d$ is large and $p$ is close to the percolation threshold, $Z_N$ does not grow as fast as its mean value.
Our result apparently contradicts some of the prediction made in the litterature. For instance in  \cite{cf:H2}, \cite{cf:CR} and many pait is claimed that 
disorder never modifies the behavior of the trajectories above the percolation threshold. The
reason for this contradiction might be that some author may have considered the ``annealed model'' (where averaging with respect to the percolation cluster 
is performed) which is mathematically trivial.

\medskip

On the other-hand our result  agrees with the predictions present in \cite{cf:LDM}.
Furthermore according to Le Doussal and Machta \cite{cf:LDM}, 
this localization phenomenon does not occur for small dilution, i.e.\ when $p$ is close to one. 
This means that there is a non-trivial transition in $p$ separating a phase where trajectories are delocalized from one where the are delocalized. 
On the mathematical level, this remains a challenging conjecture.

\subsection{Model and results}

We consider $\mathcal S_N$ the set of self-avoiding path of length $N$ starting from the origin in $\bbZ^d$ which is equipped 
with its usual lattice structure.
The notation $|\cdot|$ denotes the graph distance in $\bbZ^d$.

\begin{multline}
 \mathcal S_N:=\{ (S_n)_{n\in [0,N]}\ | \ S_0=0,\ \forall n\in [0,N-1], |S_{n}-S_{n+1}|=1, \\
 \forall i\ne j \in[0,N], S_i\ne S_j\}.
\end{multline}

Let $(\go_e)_{e\in E_d}$ be a field of IID Bernoulli random variable of parameter $p$ indexed by the edges of $\bbZ^d$. 
We denote by $\bbP_p$ or $\bbP$ the associated law. When $\go_e=1$ we say that the edge is open.
We say that a lattice path in $\bbZ^d$ is open if all the edges that composes it are open.

\medskip

We are interested in the asymptotic behavior of the number of self-avoiding open paths of length $N$ starting from the origin

\begin{equation}
 Z_N:=\sum_{S\in \mathcal S_N} \ind_{\{S \text{ is open for $\go$}\}}.
\end{equation}

We can define an analogous quantity $Z_{N,x}$  by counting open  self-avoiding path starting from $x\in \bbZ^d$ instead of the origin.
Of course $Z_{N,x}$ is equal to zero for large $N$ if $x$ lies in a finite cluster of open edges.
For the reason, we focus only on the supercritical regime of percolation $p>p_c$ ($p_c$ denoting the percolation threshold) 
where a.s.\ there is a unique infinite connected component of open edges
(see e.g.\ the first chapters of \cite{cf:Grimm} for an introduction to percolation).
We call $\cC$ the unique infinite percolation cluster.

\medskip

In  \cite{cf:conn}, we introduced the notion of connective constant for the dilute lattice.
Recall that $\mu_d=\mu_d(1)$ the connective constant of $\bbZ^d$ is defined by 
\begin{equation}
\mu_d:=\lim_{N\to \infty} |\mathcal S_N|^{1/N}
\end{equation}
Here $|\cdot|$ denote the cardinal of a set and we will keep this notation throughout the paper as it brings no confusion.

\begin{proposition}[\cite{cf:conn} Proposition 1.1]\label{taex}
For $p>p_c$, for every $x\in \mathcal C$, the limit
\begin{equation}\label{limsupp}
 \limsup_{N\to \infty} (Z_{N,x})^{\frac 1 N},
\end{equation}
does not depend on $x$ and is $\bbP$-a.s.\ constant.
We call this limit the \textit{quenched} connective constant of the dilute lattice and denote it by $\mu_d(p)$.
It satisfies the inequality
\begin{equation}\label{annqu}
 \mu_d(p)\le p \mu_d,
\end{equation}
where 
$$\mu_d:=\mu_d(1)=\lim_{N\to \infty} |\mathcal S_N|^{1/N},$$
is the connective constant of $\bbZ^d$.
We call
\begin{equation}
 \bbE_p \left[Z_N\right]^{\frac1N}= p\mu_d(1).
\end{equation}
the \textit{annealed} connective constant.

\medskip

Moreover, the ratio between quenched and annealed connective constant $$\mu_d(p)/p\mu_d(1)$$ is a 
non-decreasing function of $p$ on $(p_c,1]$.

\medskip

In particular there exists $p_c^{(2)}\in[p_c,1]$ such that 
$$\mu_d(p)<p\mu_d(1) \text{ if } p<p_c^{(2)} \quad \text{and} \quad \mu_d(p)=p\mu_d(1) \text{ if } p>p_c^{(2)}.$$

\end{proposition}

\begin{rem}\rm
The notion of connective constant for non-transitive graphs has also recently been studied in a non-random setup ; in \cite{cf:GrimmLi} a universal 
lower-bound is given for the
for the connective constant of $d$-regular graphs.
\end{rem}

Whether the inequality \eqref{annqu} is strict or not is related, at least on a heuristic level to trajectorial properties of
the self-avoiding walk on $\mathcal C$. The Self-Avoiding Walk of length $N$ on $\mathcal C$ is the process defined on $\mathcal S_N$ 
by the probability law

\begin{equation}
\pi^\go_N(S):=\frac{1}{Z_N}\ind_{S \text{ is open }}.
\end{equation}

The general idea is that if $Z_N$ behaves asymptotically like its expected value, it means that the environment is spatially averaging so that
the self-avoiding walk in the inhomogeneous medium keeps the features of self-avoiding walk on $\bbZ^d$ (see e.g.\ \cite{cf:Flo, cf:Slade} for conjectures and mathematically proved results).
On the other hand if $Z_N\ll \bbE[Z_N]$, and \textit{a fortiori} if the two quantities have a different behavior on the exponential scale (\textit{i.e.} if $\mu_d(p)< p \mu_d(1)$)
one should observe localization of the trajectories under $\pi^\go_N$ 
(see \cite{cf:Com} and references therein for the case directed polymers, and \cite{cf:LDM} for physicists prediction).

\medskip

An important issue is then to decide whether a genuine phase transition occurs at $p_c^{(2)}$ \textit{i.e.} if $p^{(2)}_c\notin \{p_c,1\}$.
In \cite{cf:conn}, we proved that in dimension $2$, the inequality \eqref{annqu} is always strict for all values of $p<1$ so that there is no phase transition ($p^{(2)}_c=1$). Le Doussal and Machta 
\cite{cf:LDM} believe that this is the case 
also for $d=3$, and we agree with this conjecture although it should be difficult to prove with available tools 
(see the introduction of \cite{cf:conn} for more discussions).
On the contrary, when $d\ge 4$, the same authors conjectured that for small edge dilution, the disorder is irrelevant: in other words that for $p$ close to one $\mu_d(p)=p\mu_d(1)$ or that  $p^{(2)}_c<1$. Note that the conjectures in \cite{cf:LDM} are not formulated for the connective constant 
but in terms of trajectorial properties.

\medskip

In this paper we focus on the question: is $p^{(2)}_c>p_c$ in general?  We give a positive answer to this question in high dimension.
The strategy we adopt to prove such a result presents some similarities with the one we used in \cite{cf:HDP} to solve an analogous question for 
the problem of oriented percolation: we prove an asymptotic lower bound on $p_c^{(2)}$ which is larger than the asymptotic development in $(2d)^{-1}$ of
$p_c$.

\medskip

The asymptotic development for the percolation threshold up to the third order has been computed by Hara and Slade \cite{cf:HS} using the lace expansion.
They proved that
$$p_c:=\frac{1}{2d}+\frac{1}{(2d)^2}+\frac{7}{2}\frac{1}{(2d)^3}+O(d^{-4}).$$
See also \cite{cf:SVdH} for a related result and a bibliography on this topic.

\medskip

Let us present now our main result 

\begin{theorem}\label{mainres}
 Given $\gep>0$, when $d$ is large enough and 
 $$p\le \frac{1}{2d}+\frac{1}{(2d)^2}+(2+3\log 2-\gep)\frac{1}{(2d)^3}$$
 then
 \begin{equation}
  \mu_d(p)<p\mu_d.
 \end{equation}
 
As a consequence when $d$ is large enough, 
$$p_c^{(2)}> p_c.$$
\end{theorem}

\subsection{Ideas behind the proof and interpretation of the result} 

Our method to obtain a lower bound on $p_c^{(2)}$ relies on three ideas:
the first one is to consider the size biased measure  where the 
probability of a given environment is proportional to $\bbP(\go)Z_N(\go)$.
This gives us a nice characterization of the strong disorder regime $p<p_c^{(2)}$: see Lemma \ref{sizebi}.

\medskip

The second idea is to use a special construction of the size biased measure sometimes referred to as the {\sl spine construction}.
It says that the environment under the size-biased law can be obtained by opening a self avoiding path of length $N$ chosen uniformly at random (the spine), 
and opening every other edges in $\bbZ^d$ independently with probability $p$: see Lemma \ref{spine}

\medskip

Finally, and this is the most important step, we use this spine construction 
to show that  additions of edges around the spines can generate a lot of open paths, 
showing that the size-biased measure and the original one are very different when $p$ is small (Proposition \ref{mainprop}).

\medskip

All these three ideas were used in an earlier work concerning oriented percolation \cite{cf:HDP}, 
and the way to proceed for the two first step is exactly identical. However, here the final step is much
more involved and requires some heavy-machinery for several reasons:  the possibility of interaction between the walk and its past, 
the modification of the length when adding "bridges on the walk" etc...

\medskip

At first sight, after reading our proof, one could think that the asymptotic lower-bound we find for $p_c^{(2)}$ 
is larger than $p_c$
for accidental reasons and that it might not be satisfied for some other type of high dimensional lattice. 
This is however not the case: the fundamental reason for the difference between the two is that the contribution of ``square of open edges'' (i.e.\ of 
regions where four edges forming a square are open)
is more important under the size-biased measure than under the original one. In the high dimensional limit, the contribution of these open square gives the leading asymptotic term for  $(Z_N)^{1/N}$.
Thus our result can be expected to be true for percolation models in the mean-field limit (see for instance \cite{cf:HDP} where the same result holds for directed percolation essentially for the same reasons).
We believe that $p^{(2)}_c>p_c$ in all dimensions but we currently have no argument to justify it for finite $d$.
Note finally that we did not try to explicit what $d$ large enough in Theorem \ref{mainres} means. The reason for that is that 
quantitative estimates would be quite hard to derive and in any case, very suboptimal.

\section{Decomposing the proof of Theorem \ref{mainres}}

\subsection{The size biasing}

The size biazing technique consist in studying the process under a new measure, which gives a larger probability to 
environments $\go$ for which the partition function $Z_N$ is larger.
We introduce $W_N$ the renormalized partition function

\begin{equation}
 W_N(\go):=\frac{1}{p^N|\mathcal S_N|}Z_N(\go).
\end{equation}
Note that the inequality \eqref{annqu} is strict if and only if $W_N$ decays exponentially fast.

\medskip

We define $\tilde \bbP$ the size biased-law, which is absolutely continuous with respect to $\bbP$ and whose derivative is
\begin{equation}
 \frac{\dd \tilde \bbP}{\dd \bbP}(\go):= W_N(\go).
\end{equation}

The exponential decay of $W_N(\go)$ under $\bbP$ corresponds more or less to its exponential growth under $\tilde \bbP$.
This the content of the following Lemma.

\begin{lemma}\label{sizebi}
If there is a constant $c$ satisfying  
\begin{equation}\label{lhypo}
 \tilde \bbP[ W_N\le e^{cN} ]\le e^{-cN}, 
\end{equation}
then   $\mu_d(p)<p\mu_d$.
\end{lemma}

\begin{proof}
The result and proof are similar to Proposition 4.2 in \cite{cf:HDP}, and we include the proof here just for the sake of completeness.
The inequality $\mu_d(p)<p\mu_d$ is equivalent to the almost sure exponential decay of $W_N$.

\medskip

Assume that \eqref{lhypo} holds. Then
from the definition of $\tilde \bbP$ we have
\begin{multline}
\bbP[ W_N\in [e^{-cN/2},e^{cN}]]
=\tilde \bbE\left[(W_N)^{-1}\ind_{W_N\in [e^{-cN/2},e^{cN}]}\right]
\\
\le e^{c/2N}\tilde \bbP[ W_N\le e^{cN} ]\le e^{-cN/2}.
\end{multline}
The Markov inequality implies that $\bbP[ W_N\ge e^{cN}]\le e^{-cN}$, and hence 
by the Borel-Cantelli Lemma, we have
\begin{equation}
 \limsup_{N\to \infty} (W_N)^{1/N}=e^{-c}.
\end{equation}

\end{proof}

\subsection{The spine method}

The size biased measure can be obtained from the original one by opening all the edges along a randomly chosen self avoiding path.
Similar constructions have been used for quite a while for the study of branching structure (see \cite{cf:LPP}) or directed polymers \cite{cf:Birk}.

\medskip

Given a path $S\in \mathcal S_N$ and an environment $\go$ we define the environment $\tilde \go(S,\go)$ with spine $S$ by
\begin{equation}
\tilde \go_e:=
\begin{cases}
1 \quad \text{ if } e\in S,\\
\go_e \quad \text{ if } e\notin S.
\end{cases}
\end{equation}

Then set 

\begin{equation}
\tilde Z_N(S,\go):= \sum_{S' \in \mathcal S_N} \ind_{S' \ \text{ is open for $\tilde \go$}}.
\end{equation}

\begin{lemma}\label{spine}
The law of $Z_N(\go)$ under $\tilde \bbP$, is the same as the law of 
$\tilde Z_N(S,\go)$ under $\pi_N\times \bbP$.
\end{lemma}
\begin{proof}
See e.g.\ \cite{cf:Birk} Lemma 1.
\end{proof}

\subsection{A lower bound on path counting}

The most important step, and whose proof will be the focus of all the rest of the paper is to show that
as soon as $p\ge (2d)^{-1}$, with large probability, a lot of paths are open. The idea is to use the spine as a backbone to 
construct a lot of open paths with large probability. This result combined with the two Lemmata 2.1 and 2.2 easily yields Theorem \ref{mainres}.

\begin{proposition}\label{mainprop}
Given $\gep>0$,
when $d$ is large enough,  there exists $c(\gep,d)$ such that   for all $p\ge (2d)^{-1}$, for all $N$ large enough

\begin{equation}\label{acrimed2}
\pi_N\otimes \bbP \left[ \tilde Z_N(\go,S)\ge 2^{\frac{(3-\gep)}{(2d)^{2}}N}\right]<e^{-c(\gep,d)N}.
\end{equation}
\end{proposition}

\begin{proof}[Proof of Theorem \ref{mainres} from Proposition \ref{mainprop}]
According to Lemma \ref{spine} and Lemma \ref{sizebi}, the Theorem reduces to prove that there exists 
$c$ such that 
\begin{equation}\label{crocodiles}
\pi_N\otimes \bbP \left[ \tilde Z_N(\go,S)\ge e^{cN}p^N|\mathcal S_N| \right]<e^{-cN}.
\end{equation}
 From the Proposition \ref{mainprop},  and the fact that (cf.\ the definition of the connective constant)
 $$|\mathcal S_N|=(\mu_d+o(1))^N,$$
Equation \eqref{crocodiles} is satisfied for $N$ large enough if 
 
 \begin{equation}\label{cocorico}
p<(\mu_d)^{-1} 2^{\frac{(3-\gep)}{(2d)^{2}}}.
 \end{equation}
We use the following asymptotic expansion in $d$ of $\mu_d$ from \cite{cf:Kes}
  \begin{equation}\label{cracoucas}
 \mu_d:=2d-1-\frac{1}{2d}+O(d^{-2}),
 \end{equation}
 to get that, for $d$ large enough, the inequality \eqref{cocorico} is satisfied if
 \begin{equation}
 p\le \frac{1}{2d}+\frac{1}{(2d)^2}+(2+3\log 2-\gep)\frac{1}{(2d)^3}.
 \end{equation} 
 \end{proof}

\section{Strategy of proof for Proposition \ref{mainprop}}

\subsection{Bridges over the spine}

Our strategy is to have a lower bound is to look only at paths that uses either edges of the spine or make very short excursions out of it
(we will call these excursions \textsl{bridges}).
We first  have to restrict ourselves to a set of good spines $S$ and show that most spines are good.
Then we show that when a spine is good, with large probability, we have a lot of open bridges, which allows to have a lot of open paths.

\begin{figure}[hlt]
 \begin{center}
 \leavevmode %\epsfysize =5 cm
 \epsfxsize =14 cm
 \psfragscanon
 \psfrag{a}{$(a)$}
 \psfrag{b}{$(b)$}
\psfrag{c}{$(c)$}
 \epsfbox{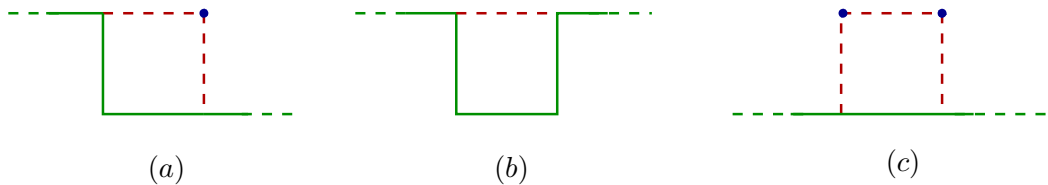}
 \end{center}
 \caption{\label{threebridges} Three possible ways of constructing a bridge (dashed edges) over the spine. The thick dots indicate the free sites.}
 \end{figure}

Let us introduce some notation. We set $X_n:=S_{n}-S_{n-1}$ to be the $n$-th increment of $S$,
$U_N$ to be the set of indices when the spine makes a U-turn,
\begin{equation}\label{un}
 U_N:=\{n\in[3,N]\ | \ S_{n}\sim S_{n-3}\},
\end{equation}
and $T_N$ the set of indices around which $S$ does not bend
\begin{equation}\label{tn}
T_N:=\{n\in[1,N-1]\ | \ X_n= X_{n-1}\}.
\end{equation}
Let $I=\{x \in \bbZ^d \ | \ |x|=1 \ \}$ to be the set of nearest neighbors of the origin.

We have to consider three types of bridges (figure \ref{threebridges}):
\begin{itemize}
 \item [(a)] The edges $(S_{n-1},S_{n-1}+X_{n+1})$ and  $(S_{n-1}+X_{n+1},S_{n+1})$ are $\go$ open for some 
 $n\notin T_n$. We call this an $(a)$-bridge over $S_n$.
 \item [(b)] The edge $(S_{n},S_{n-3})$ is open for some $n\in U_n$, we call this a $(b)$-bridge linking $S_n$ to $S_{n-3}$. 
 \item [(c)] The edges $(S_{n-1},S_{n-1}+e)$, $(S_{n-1}+e,S_{n}+e)$ and $(S_{n}+e,S_n)$ are open for some $e\in I \setminus \{ X_n, -X_{n-1}, X_{n+1}\}$,
 we call this a $(c)$-bridge over the edge $(S_{n-1},S_n)$ (the restriction of the direction $e$ is there because we do not want our bridge to use edges of the spine).
 \end{itemize}

To a bridge we can associate a square of open edges in $\tilde \go$ which is composed of $2$, $3$, or $1$ edges of the spine and the rest forms the bridge (see Figure \ref{threebridges}).
We call the vertices of that are at the end of two edges of the bridge the free sites (there is one free site for an $(a)$-bridge and two for 
a $(c)$-bridge and none for a $(b)$-bridge).

\begin{figure}[hlt]
 \begin{center}
 \leavevmode %\epsfysize =5 cm
 \epsfxsize =14 cm
 \psfragscanon
 \psfrag{a}{$(a)$}
 \psfrag{b}{$(b)$}
\psfrag{c}{$(c)$}
 \epsfbox{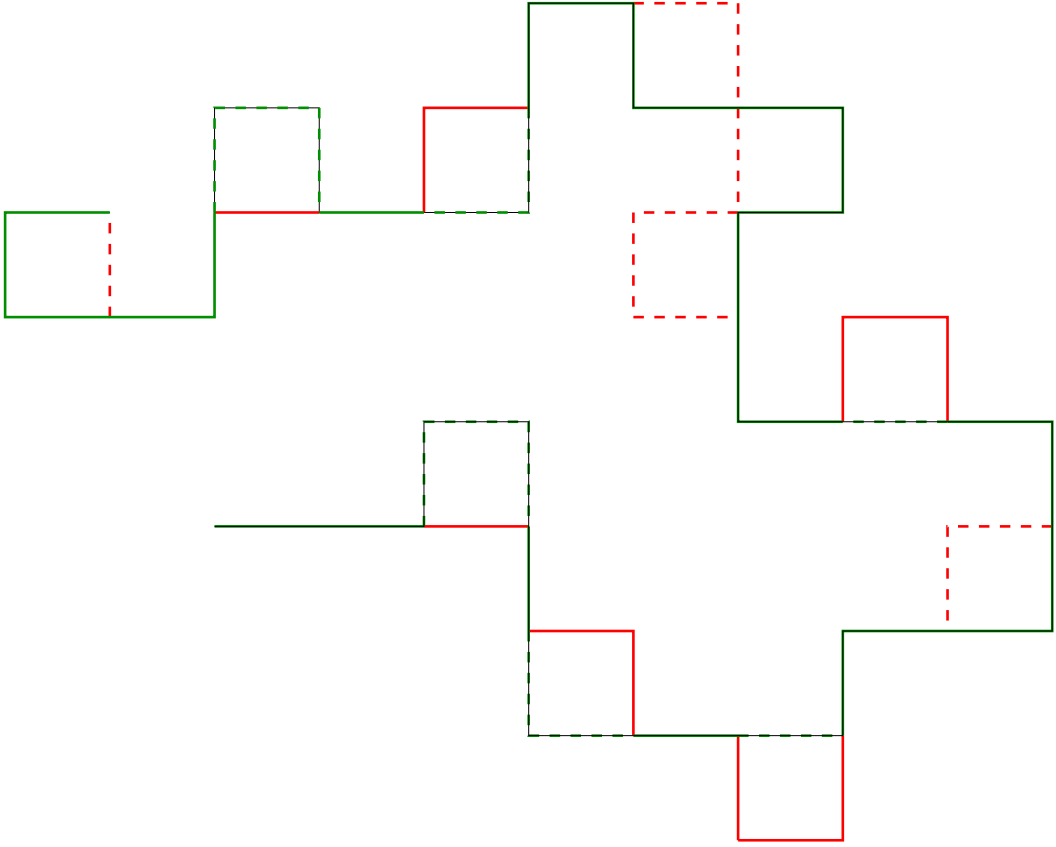}
 \end{center}
 \caption{\label{superpath} The green line represents the spine $S$, and the red one the open bridges. We can construct a path of length $N$ (continuous line) by using some of the bridges and in particular the same number of bridges of type $(b)$ and $(c)$. The unused bridges and unused parts of the spine are left as dashed lines. This construction is possible as soon as the squares formed by the bridges do not overlap i.e.\ do not share edges or free sites in common (no that we allow different squares to share some spine sites).}
\end{figure}

It can be remarked that if we have $n_a$ bridges of type $(a)$, $n_b$ bridges of type $(b)$ and $n_c$ bridges of type $(c)$ 
such that the corresponding squares do not share edges nor free sites like on figure \ref{superpath} where $n_b=4$ and $n_a=n_c=3$ (see on figure \ref{overlapbri} 
examples of configurations where bridges do overlap), then it is possible to construct

\begin{equation}\label{lowb}
N_{n_a,n_b,n_c}:=2^{n_a} \sum_{k=0}^{n_b\wedge n_c} \binom   {n_b} k \binom  {n_c} k \ge 2^{n_a} \frac { 4^{n_b\wedge n_c}}{C(n_b\wedge n_c)},
\end{equation}
distinct self-avoiding paths of length $N$ that are open for $\tilde \go$. 
This is because if one selects a subset of the $(a)$-bridges and subsets of the $(b)$ and $(c)$-bridges of the same cardinality one can form an open path
for $\tilde \go$ that uses the edges of the selected bridges and the edges of the spine elsewhere else like on Figure \ref{superpath} (we have to use the same number of $(b)$ and 
$(c)$ because these type of bridges modify the total length of the path
and we want the length to be equal to $N$). Obviously different choices for the set of bridges give a different path.

\begin{figure}[hlt]
 \begin{center}
 \leavevmode %\epsfysize =5 cm
 \epsfxsize =14 cm
 \psfragscanon
 \psfrag{a}{$(1)$}
 \psfrag{b}{$(2)$}
\psfrag{c}{$(3)$}
 \psfrag{d}{$(4)$}
 \psfrag{e}{$(5)$}
\psfrag{f}{$(6)$}
 \epsfbox{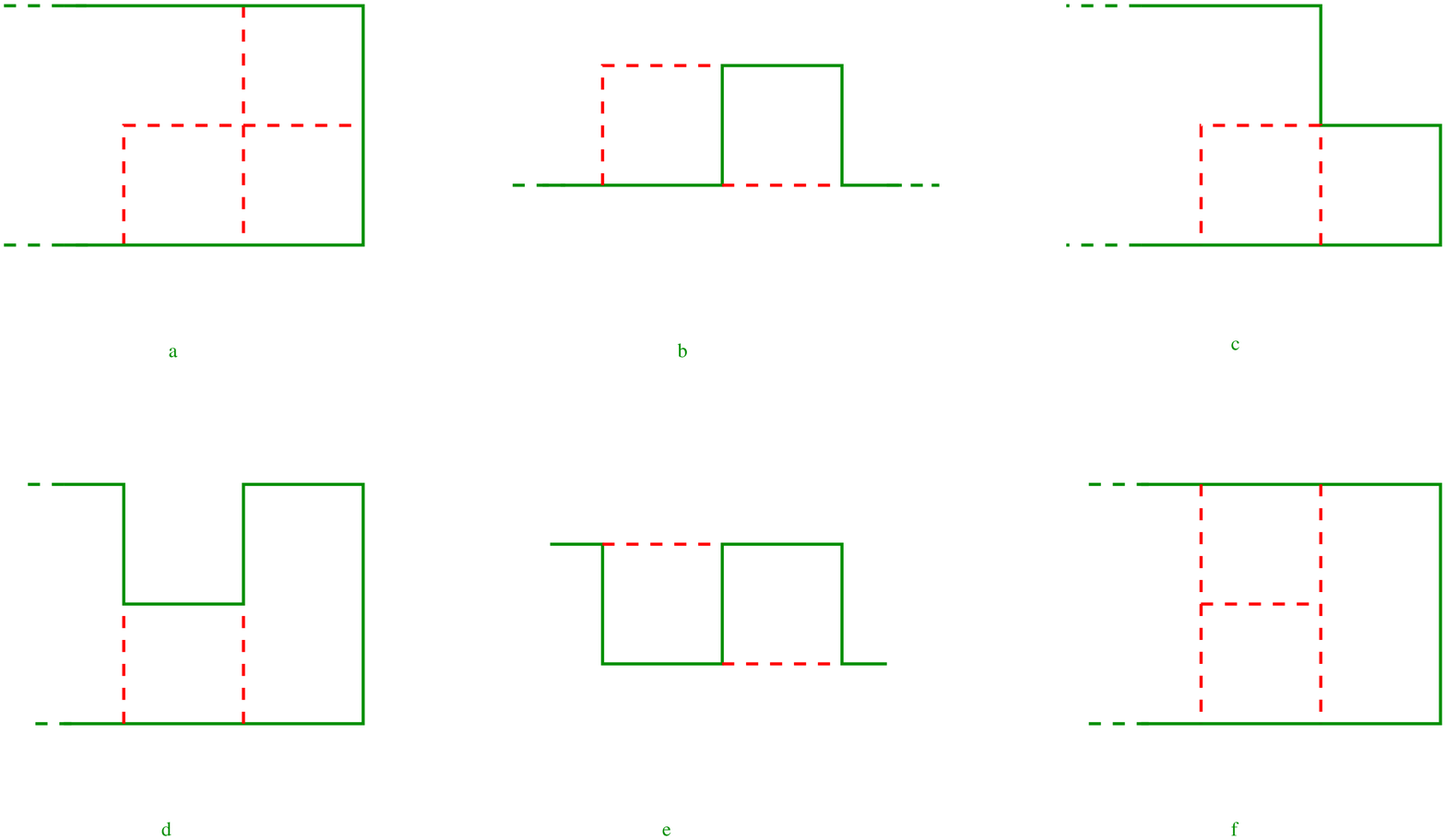}
 \end{center}
 \caption{\label{overlapbri} 
 Here we represent some kind of situations were bridges cannot be used because they overlap with the spine or with another bridge.
 We will have to be careful in our construction to avoid these kind of situation.}
 \end{figure}

Thus our work will be focused on proving that with an overwhelmingly large probability we can construct a lot of bridges 
with squares that overlap only on spine sites.

\medskip

The typical value of the number of bridges of type $(a)$, $(b)$ and $(c)$ at the first order is not difficult to derive at the heuristic level if we accept that 
in high dimension the spine behaves like a simple random walk: given $n \notin T_N$ (and in large dimension the proportion of indices not in $T_N$ tend to one),
the probability of a bridge of type $(a)$ being open is $(2d)^{-2}$ so that
the typical number of bridges of type $(a)$ is $[(2d)^{-2}+o(d^{-2})]N$. For bridges type $(b)$
we notice that the probability of making a $U$-turn for the self-avoiding walk should be roughly the same as for the simple random walk
where it is asymptotically equivalent to $1/2d$ when $d$ is large. Hence we typically have $|U_N|=[(2d)^{-1}+o(d^{-1})]N$ and each for each $n$ in $U_N$,
we have a probability $(2d)^{-1}$ to have an open bridge, giving a heuristic number of $(b)$-bridges of $[(2d)^{-2}+o(d^{-2})]N$.
For $(c)$, we notice that for each edge, there are about $2d$ directions in which a bridge of type $c$ can be open. Considering that each bridge is composed of three edges
and has a probability 
$(2d)^{-3}$ of being open, we get a similar asymptotic for the number of $(c)$-bridges.
Hence from \eqref{lowb} one should have typically

\begin{equation}\label{acrimed}
Z_N(\go,S)\ge 2^{(3(2d)^{-2}+o(d^{-2}))N}.
\end{equation}

However, to make this rigorous, we have to show that on the complement of a set of exponentially small probability, we 
can construct large sets of bridges that do not overlap with each other. Thus we would have that \eqref{acrimed} holds with large probability.

\subsection{Good spines and good $\go$}

The remainder of the paper is devoted to that purpose. We can separate the proof of Proposition \ref{mainprop} in two steps.
The first step is to reduce to a set of good spines for our construction, on which
were there are lots of spaces to construct bridges that do not overlap. We have to show that the set of good spine has a large probability.

\medskip

The second step is to show that, given a good spine, the probability of being able to construct a lot of non-overlapping bridges is large.
Introducing these two results requires some notation.

\medskip

Set $V^1_N$ resp.\ $V^2_N$ to be the set of time for which the spine never comes back at a distance one (resp.\ two) after more than one (resp.\ two)step, in both directions.

\begin{equation}\begin{split}
 V^1_N&:=\{n\in[0,N] \ | \forall m\le N, |m-n|>1 \Rightarrow |S_m -S_n|>1\},\\
 V^2_N&:=\{n\in[0,N] \ | \forall m\le N, |m-n|>2 \Rightarrow |S_m -S_n|>2\},\}. 
\end{split}\end{equation}
Note that when $N$ is larger than $6$ we have $V^{2}_N\subset V^{1}_N$.
Given $\gep>0$, we define $\mathcal A_{\gep,N}=\mathcal A$ the set of good spines
\begin{equation}
 \mathcal A_{\gep,N}:= \{ |V^2_N|\ge(1-\gep)N\}\cap  \{|T_N|\le \gep N\}
 \cap\left\{|U_N|\in \left[\frac{(1-\gep)N}{2d},\frac{(1+\gep)N}{2d}\right]\right\}.
\end{equation}

Then, we divide Proposition \ref{mainprop} into two results.

\begin{proposition}\label{gudespine}
For fixed $\gep$, for $d$ large enough, there exists $c(\gep,d)>0$ such that for all $N$ large enough
\begin{equation}
 \pi_N(\mathcal A)\ge 1- e^{-c(d,\gep)N}.
\end{equation}
\end{proposition}

\begin{proposition}\label{bridgeconstructor}
For fixed $\gep$, for $d$ large enough, there exists $c(\gep,d)>0$ such that for $p\ge (1/2d)$,
for $N$ is large enough, and $S\in \mathcal A$ then
\begin{equation}
 \bbE_p\left[ \tilde Z_N(\go,S)\le 2^{(3-6\gep)(2d)^{-2}N}  \right] \le e^{-c(\gep,d)N}. 
\end{equation}
\end{proposition}

The remainder of the paper is divided in two Section, each of which is devoted to the proof of one of the above propositions.

\section{Proof of Proposition \ref{bridgeconstructor}}
By monotonicity it is sufficient to prove the result for $p=(2d)^{-1}$. Hence in this whole section $\bbP=\bbP_{(2d)^{-1}}$.

\medskip

The strategy is exactly the one exposed above: we want to show that with a large probability there are a lot of bridges of each type and that moreover they are located at the right places.
We start with the bridges of type $(b)$.
We define $B_N$ as a subset of the $n$ such that there is an open bridge of type $(b)$ linking $S_n$ and $S_{n-3}$, more precisely

\begin{equation}
B_N:= \{n\in U_N \ | \ (S_{n},S_{n-3}) \text{ is open for $\go$ }, (S_{n-2},S_{n-5})\ \text{is closed for $\go$}\},
\end{equation}
where the second condition has to be checked only if $n-2\in U_N$ (otherwise $(S_{n-2},S_{n-5})$ is not an edge in $\bbZ^d$).
This condition ''$(S_{n-2},S_{n-5})$ is closed for $\go$" is present so that all the bridges in our set are well separated to avoid  a situation such as $(5)$ on Figure \ref{overlapbri} where the two consecutive open bridges cannot be used simultaneously.
We show that when the spine is good, $|B_N|$ cannot be much smaller than $N/(2d)^2$ or more precisely

\begin{lemma}\label{badb}
For any fixed $\gep$, for large enough $d$, there exists $c(d,\gep)$
such that for all $N$ large enough, for all $S\in \mathcal A$

\begin{equation}
\bbP [|B_N|\le (1-2\gep)N/(2d)^2]\le e^{-c(d,\gep)N}.
\end{equation}
\end{lemma}

\begin{proof}

Given the spine, $|B_N|$ depends only on the state of the edges $\{ (S_n,S_{n-3})_{n\in U_N}\}$. Furthermore changing the state of one edge changes 
$|B_N|$ by at most one.
Thus $|B_N|$ is a $1$-Lipshitz function of $\{0,1\}^{|U_N|}$ equipped with the Hamming distance. Thus we can apply McDiarmid's inequality (see \cite{cf:McD} or \cite{cf:Mos}) and get that for all positive $x$

\begin{equation}
\bbP\left[ |B_N|-\bbE[|B_N|]\le -x \right]\le \exp(-2x^2/|U_N|).
\end{equation}
The result then follows by choosing 
$$x= \bbE\left[|B_N|\right]-(1-2\gep)N/(2d)^2.$$
We just have to check that $x$ is larger than some constant times $N$.
We do so by noticing that as $|U_N|\in \left[\frac{(1-\gep)N}{2d},\frac{(1+\gep)N}{2d}\right]$ (from the assumption $S\in \mathcal A$), and that for
$n\in U_n$, $\bbP(n\in B_N)$ is either equal to $(2d)^{-1}(1-(2d)^{-1})$ or $(2d)^{-1}$ and hence
\begin{equation}
\bbE[|B_N|]\ge |U_N|(2d)^{-1}(1-(2d)^{-1})\ge (1-3\gep/2)N/(2d)^2,
 \end{equation}
 if $d$ is large enough.
\end{proof}

Now, we deal with bridges of type $(a)$. For reasons exposed earlier we do not want our bridges to overlap with briges of type $(b)$ and thus we have to choose their location carefully. We define first a set $A^0_N$ where we can build the bridges. We call
 
\begin{equation}
A^0_N:=\{n\in [1,N-1] \ |\ (n-1)\in V_N^2,  n\notin T_N,\  [n,n+3]\cap U_N=\emptyset \},
\end{equation}
the condition $[n,n+3]\cap U_N=\emptyset$ is there to avoid interaction of bridges of type $(a)$ over $n$, with $n\in A^0_N$
with bridges of type $(b)$ (like case $(2)$ in Figure \ref{overlapbri}).

Then for $n\in A^0_N$, let $\chi_n$ be the variable which denotes if the bridge of length two that goes over $n$ is open:
\begin{equation}\label{chidef}
\chi_n:=\ind\{(S_{n-1},S_{n-1}+X_{n+1}) \text{ and }  (S_{n+1},S_{n-1}+X_{n+1}) \text{ are open for $\go$ }\}.
\end{equation}

 \medskip
 
 Then we set $A_ N$ to be our set of well located bridges of type $(a)$

\begin{equation}
A_N:=\{n\in  A^0_N \ | \chi_n=1 \text{ and }  \chi_{n-1}\ind_{\{(n-1)\in A^0_N\}}=0\}.
\end{equation}
Again the condition $\chi_{n-1}=0$ is there to avoid interaction between different bridges (here two bridges of type $(a)$). 
We prove 
\begin{lemma}\label{bada}
For any fixed $\gep$, for large enough $d$, there exists $c(d,\gep)$
such that for all $N$ large enough, for all $S\in \mathcal A$
\begin{equation}
\bbP [|A_N|\le (1-3\gep)N/(2d)^2]\le e^{-c(d,\gep)N}.
\end{equation}
\end{lemma}

\begin{proof}

Given the spine, we  claim that the $(\chi_n)_{n\in A^0_N}$ are IID Bernoulli of parameter $(2d)^{-2}$.
The only thing to check it that there is no over overlap between the different bridges over $n$, $n\in A^0_N$.
For an edge to be used by two distinct bridges of type $(b)$, it must be for two bridges over $S_n$ and $S_{n+2}$.
For  an overlap to occur we need that the free sites of the two bridges to be the same i.e.\ $S_{n+1}+X_{n+3}=S_{n-1}+X_{n+1}$, or equivalently
$X_n=-X_{n+3}$. 
This implies $|S_{n-1}-S_{n+3}|=2$, and thus $(n-1)\notin V^2_N$ and hence $n\notin A^0_N$. For this reason, bridges over $S_n$,  $n\in A^0_N$  are edge disjoint.

\medskip

Given the spine $S$, $|A_N|$ 
is a $1$-Lipshitz function of $(\chi_n)_{n\in A^0_N}$ for the Hamming distance on $\{0,1\}^{|A^0_N|}$. Thus
applying McDiarmid's inequality \cite{cf:McD} (as the $\chi_n$ are IID)

\begin{equation}\label{troll}
\bbP\left( |A_N|-\bbE[|A_N|]\le -x\right)\le e^{-2x^2/|A^0_N|}.
\end{equation}
The result then follows by choosing 
$$x= \bbE\left[|A_N|\right]-(1-3\gep)N/(2d)^2.$$
It remains to check that this $x$ is proportional to $N$ by giving a lower bound on $\bbE[|A_N|]$.
By independence of the $\chi_n$, given the spine, for $n\in A_N^0$, $\bbP[n\in A_N]$ is either 
$\left(1-(2d)^{-2}\right)(2d)^{-2}$ if $n-1\in A^0_N$ or $(2d)^{-2}$ if not. Hence

\begin{equation}
\bbE[|A_N|]\ge |A_N^{0}|\left(1-(2d)^{-2}\right)(2d)^{-2}.
\end{equation}
For $S\in \mathcal A$
\begin{multline}\label{agro}
 |A_N^{0}|=\left|[1,N-1]\cap(V^2_N+1)\setminus\left(T_N\cup \left(\bigcup_{i=0}^3 U_N-i\right)\right)\right|\\
 \ge N-1-(N-|V^N_2|)-|T_N|-3|U_N|\ge N(1-3\gep/2),
\end{multline}
for large enough $d$.
Hence when  $d$ large enough
\begin{equation}
\bbE[|A_N|]\ge N(1-2\gep)(2d)^{-2},
\end{equation}
which is enough to conclude.
\end{proof}

Finally we treat the case of bridges of type $(c)$.
We do not want these bridges to overlap with eachother nor with bridges of type $(b)$ and $(a)$ and thus we must be careful about their location.
Define as before a set $C^0_N$ where we allow construction of bridges.
\begin{equation}
C^0_N:=\{n\in [0,N-1] \ | \{n,n+1\}\subset  V^2_N,\  [n+1,n+3]\cap U_N=\emptyset \}.
\end{equation}
The condition $\{n,n+1\}\subset  V^2_N$ is present because we do not want our bridges to overlap with some other distant bridge 
(of type $(a)$ or $(c)$ see $(1)$ or $(4)$ of figure \ref{overlapbri}),
it also guarantees that the bridges over $(S_n,S_{n+1})$ does not use edges or sites on the spine.  The condition
 $[n+1,n+3]\cap U_N=\emptyset$ is there to avoid interaction with bridges of type $(a)$.
 
However a bridge over the edge $(S_n,S_{n+1})$ and another bridge over the edge $(S_{n+1},S_{n+2})$ may overlap. 
We want to avoid this and set our definition of 
the set of good bridges accordingly.  For $n$ set $\xi_n$ the event that there exists an open bridge over $(S_n,S_{n+1})$ (recall that $I$ is the set of nearest neighbor in $\bbZ^d$)
\begin{multline}
\xi_n:=\ind\{\exists e\in I \setminus  \{-X_{n-1},-X_{n},X_{n+1},X_{n+2},X_{n+3}\} \\
\text{ $(S_{n},S_{n}+e)$, $(S_{n}+e,S_{n+1}+e)$ and $(S_{n+1}+e,S_{n})$ are open}\}.
\end{multline}
Note that the possibilities for the direction $e$ have been reduced to avoid overlap with bridges of type $(a)$ .
Define then the set of good bridges (recall \eqref{chidef})

\begin{equation}
C_N:=\{n\in C^0_N\ | \ \xi_n=1,\xi_{n-1}=\chi_n=\chi_{n+1}=0
\}.
\end{equation}
The condition $\xi_{n-1}=\chi_n=\chi_{n+1}=0$ is there to ensure that any open bridge that is selected does not overlap with one previously chosen in $C_N$ or $A_N$.
We want to prove

\begin{lemma}\label{badc}

For all $\gep$, for all $d$ large enough,
there exists $c(d,\gep)>0$ such that  for all $N$ large enough and
$S\in \mathcal A$, 
\begin{equation}
\bbP [|C_N|\le (1-5\gep)N/(2d)^2]\le e^{-c(d,\gep)N}.
\end{equation}
\end{lemma}
\begin{proof}

A computation similar to \eqref{agro} shows that for a fixed $\gep$, when the spine $S$ belongs to $\mathcal A$ and $d$ large enough
$$|C^0_N|\ge (1-3\gep)N.$$

For $n\in C_N^0$, $\bbP\left[n\in C_N\right] $ is not straight-forward to compute as it depends on the local configuration of the spine but one can easily get a bound on it. Notice
that by definition $\chi_n$ and $\chi_{n+1}$ are independent of eachother and of $(\xi_n,\xi_{n-1})$.
For $\xi_{n}$ and $\xi_{n-1}$ to be one simultaneously, there are two possibilities, either $5$ edges 
forming two joint bridges are open: there are at most $(2d)$ possibilities for this to occur, each one of probability $(2d)^{-5}$; 
or two disjoint bridges are open using a total of $6$ edges, there are at most $(2d)^2$ option for choosing the bridges and the probability of both being open is $(2d)^{-6}$ and 
hence using union bound for all these events we have

\begin{equation}
\bbP(\xi_n=1,\xi_{n-1}=1)\le 2(2d)^{-4}.
\end{equation}
As a consequence we have

\begin{multline}
\bbP(\xi_n=1, \xi_{n-1}=\chi_n=\chi_{n+1}=0)=(\bbP(\xi_n=1)-\bbP(\xi_n=1,\xi_{n-1}=1))(1-(2d)^{-2})^2\\
\ge 
\left(1-(1-(2d)^{-3})^{2d-5}-2(2d)^{-4}\right)(1-(2d)^{-2})^2\ge (1-\gep)(2d)^{-2},
\end{multline}
when $d$ is large enough, and

\begin{equation}
\bbE\left[|C_N|\right]\ge |C^0_N| (1-\gep)(2d)^{-2}\ge (1-4\gep)(2d)^{-2}.
\end{equation}

Given $S$, set $E_S$ to be the set of edges whose state might have an effect on $C_N$, that is
\begin{equation}
E_S:=\{ (S_n,S_n+e), n \in [0,N] \} \cup \{ (S_n+e,S_{n+1}+e), n\in [0,N-1]\}.
\end{equation}
As changing the state of one edge can only affect the state of finitely many bridges, $|C_N|$ is a Lipchitz function of 
$\{0,1\}^{E_S}$ ($\go$ restricted to $E_S$) with Lipchitz constant $K_d$.
Hence by McDiarmid's concentration inequality

\begin{equation}\label{troll2}
\bbP( |C_N|-\bbE\left[|C_N|\right]\le -x)\le e^{-x^2/ (2K_d^2|E_S|)}.
\end{equation}

Hence the result is proved by using the inequality for 
$$ x=  \bbE\left[|C_N|\right]-(1-5\gep)(2d)^{-2}N\ge \gep(2d)^{-2}N,$$
and  remarking that $|E_S|\le 4d N$.
 \end{proof}

\begin{proof}[Proof of Proposition \ref{bridgeconstructor} from Lemmata \ref{badb}, \ref{bada} and \ref{badc}]
First we check that 

\begin{equation}\label{mirex}
Z_N(\go,S)\ge 2^{|A_N|}\sum_{k=1}^{|B_N|\wedge |C_N|}\binom k{|B_N|}\binom k  {|C_N|}\ge  \frac{2^{|A_N|+2|B_N|\wedge |C_N|}}{C|B_N|\wedge |C_N|}.
\end{equation}
It is sufficient for this to exhibit an injective map from

\begin{equation}
\{(\alpha,\gb,\gga)\in \mathcal P(A_N)\times \mathcal P(B_N)\times \mathcal P(C_N)\ |\ |\gb|=|\gga|\},
\end{equation}
where $\mathcal P(A)$ is the set of subsets of $A$,
to 
$$\mathcal S_N(\tilde \go):=\{S\in \mathcal S_N\ | \ S \text{ is open for $\tilde \go$}\}.$$

The definition of the injection is rather straight-forward: to each $(\alpha,\gb,\gga)$ is associates the paths that uses the edge of the spines everywhere except
where the bridges in $\alpha$, $\gb$, $\gga$ are located (recall Figure \ref{superpath}). Because of the condition $|\gb|=|\gga|$ the length of the obtained path is indeed $N$.

\medskip

What there is to prove is that this construction is indeed possible, \textit{i.e.} that it is possible to use simultaneously all theses bridges and that the obtained path is self-avoiding
(see Figure \ref{overlapbri} for possible complications). 
Recall that when a bridge is open, it corresponds to a square of open edges for $\tilde \go$ and that free sites are sites of 
those square that do not belong to any of the square's spine edges (there is one free site for each bridge of type $a$ and two free site 
for each bridges of type $b$).

\medskip

To prove that our injection is well defined, we have to show is that two of these square never have an edge in common and that the 
free sites of a given square do not belong to another square nor to the spine. The reader can check that these assumptions are guaranteed 
by our definitions of $A_N$, $B_N$ and $C_N$. The fact that free sites do not belong to the spine is guaranteed by the fact that for 
$n$ in $A_N$, $(n-1)\notin V^1_N$, and for $n\in C_N$, $(n, n+1)\notin V^1_N$. We leave to the reader to check that there is no edge overlap thanks 
to our assumptions as this would be rather tedious to develop this point here.

\medskip

Now we combine the inequality \eqref{mirex} with the results of the three Lemmata:
when $S\in \mathcal A$, with probability larger than $1-3 e^{-c(\gep,d)}$, we have
$|A_N|\ ,|B_N|,\ |C_N|\ge (1-5\gep)N/(2d)^2=: d_N$, and hence \eqref{mirex} implies that for some appropriate constant $c$

\begin{equation}
Z_N\ge 2^{d_N}\binom{(d_N/2)}{d_N}^2\ge c 2^{3d_N}(d_N)^{-1}\ge 2^{3d_N(1-\gep)}\ge 2^{(1-6\gep)3/(2d)^2},
\end{equation}
provided that  $N$ large enough.

Let us now prove that \eqref{mirex} holds. 

\end{proof}
%\medskip
%
%
%Now note that from the definition of $A_N^0$, $C^0_N$, squares of type $(a)$ and $(c)$ cannot overlap with squares of type $(b)$ (see the condition $[n-3,n]\cap U_N=\emptyset$). And also two squares of type $(b)$ cannot overlap with one another from the definition of $B_N$.
%It remains to check that bridges of type $(a)$ and $(c)$ cannot interact.
%The fact that two $a$-square cannot share an edge is guaranteed by the condition that $A_N$ contains no pair of consecutive integer, and that $n-1\notin V^N_2$.
%
%
%
%The fact that two squares cannot overlap on an edge is guaranteed by the definitions of $A_N$, $B_N$ and $C_N$:
%the definition of $B_N$ guarantees that two squares of type $(c)$ cannot overlap, the definition of $A_N$ then guarantees that a square of type $b$ does not overlap with another square of type $b$ (two consecutive $n$ d;s,d;:q,d; :)

\section{Proof of Proposition \ref{gudespine}}

The overall strategy to prove the proposition is to prove the exponential decay (in $N$) of some bad events for walks that are Markovian and then 
prove that these Markovian walks are a good enough approximation of the self-avoiding walk (which is not Markovian).
Let us introduce notation for the Markovian approximation

Recall that $\pi_N$ to be the law of the self-avoiding walk of size $N$. Set $\pi^1$ to be the law of the non-backtracking walk, that is 
the walk for which at each step $S_{n+1}$ is chosen uniformly at random among the neighbors of $S_n$ except that the choice $S_{n-1}$ is not allowed.
Set $\pi^2_N$ the uniform law on path of length $N$ with no backtrack and no loops of length $4$ i.e.\ on
\begin{multline}
\mathcal S^4_N:= \{ (S_n)_{n\in [0,N]}\ | \ S_0=0,\ \forall n\in [0,N-1], |S_{n}-S_{n+1}|=1, \\
 \forall n \in[0,N], S_n\notin \{S_{n-2},S_{n-4}\}\}.
\end{multline}
and $\pi^2$ the Markovian nearest-neighbor walk that at each step, jumps to a random neighbor chosen uniformly at random in $$\{ x | x \sim S_n, x\ne 
\{ S_{n-1} S_{n-3}\} \  \}.$$

\medskip

Under $\pi^1$, $(S_{n})_{n\in[0,N]}$ is distributed uniformly at random among trajectories with no backtrack.
However, under $\pi^2$, it is not true that $(S_{n})_{n\in[0,N]}$ is distributed uniformly on $\mathcal S^4_N$.
On the contrary (this is easy to check), each trajectory $(S_{n})_{n\in[0,N]}$ has a probability proportional to (recall \eqref{un})
\begin{equation}\label{radpi}
\left(\frac{2(d-1)}{2d-1}\right)^{|U_{N-1}|}.
\end{equation}

We will use $\pi_1$ and $\pi^2_N$ as approximations of $\pi_N$.
The number of non-backtracking paths of length $N$ is equal to $2d(2d-1)^{N-1}$. From 
\cite{cf:FishSy}, (see also \cite{cf:Kes})
the number of paths of length $N$ with no four loops $|\mathcal S^4_N|$ satisfies 
\begin{equation}
|\mathcal S^4_N|=(1+o(1))C_{4,d} \mu^N_{4,d},
\end{equation}
where asymptotic is taken for $N\to \infty$ and $ \mu^N_{4,d}$ is a constant. Furthermore, the asymptotic development of $\mu_{4,d}$ in $d$ is given by
\begin{equation}
\mu_{4,d}:=\left(2d-1-\frac{1}{2d}+O(d^{-2})\right).
\end{equation}
Thus,  considering the asymptotic development of the connective constant $\mu_d$ \eqref{cracoucas}, one gets
that there exists a constant $C$ such that for $d$ large enough and $N\ge N_0(d)$

\begin{equation}\label{castor}
\begin{split}
 \pi^1\left((S_n)_{n\in[0,N]} \text{ is self avoiding }\right)&=\left(1-\frac{1}{(2d)^2}+O(d^{-3})\right)^N,\\
  \pi^2_N\left((S_n)_{n\in[0,N]} \text{ is self avoiding }\right)&=\left(1+O(d^{-3})\right)^N.
 \end{split}
\end{equation}

We want to show first that under $\pi_N$ the number of U-turns up to step $N$ is roughly $N/2d$ (like for the simple random walk), 
and that $|V^N_1|$ is at most twice of the same order. We also show that $|T_N|$ is small (recall \eqref{tn})
  
\begin{proposition}\label{vunetu}
For any fixed $\delta$, there exists $d_0$ such that for all $d\ge d_0$ the exists $c(d,\delta)$ such that for all $N\ge N_0(d)$,
\begin{equation}\begin{split}
 \pi_N(|T_N|\ge \delta N)&\le e^{-c(\delta,d)N},\\
  \pi_N\left(|V^1_N|\le  \left(1-\frac{1+\delta}{d}\right) N \right)&\le e^{-c(\delta,d)N},\\
  \pi_N\left(\left| |U_N|-\frac{N}{2d}\right|\ge \frac{\delta N}{2d}\right)&\le e^{-c(\delta,d)N}.
\end{split} \end{equation}
\end{proposition}
 
\begin{proof}
 The case of $T_N$ is the simplest. 
 Set  $G_n:=\ind_{X_{n}=X_{n-1}}$. Then under $\pi_1$ the $G_n$ are IID Bernoulli variables of parameter $p=1/(2d-1)$ and $|T_n|=\sum_{n=1}^N G_n$.
 Hence from Cramer's Theorem (see e.g. \cite{cf:DZ} Exercise 2.2.23) one gets that for any $x>1/(2d-1)$ and any $\gep>0$
 \begin{equation}
 \pi^1\left(|T_N|/N\ge x\right)\le \exp\left(-N(h_p(x)-\gep)/2\right),
 \end{equation}
 where
 \begin{equation}\label{hp}
h_p(x):=x\log(x/p)+(1-x)\log((1-x)/(1-p)).
\end{equation}
If $\delta$ is fixed, when $d$ is sufficiently large $h_p(x)>2\delta$, so that for$N$ large enough
 $$\pi^1\left(|T_N|/N\ge x\right)\le \exp\left(-N\delta \right).$$
 Hence the result follows (with e.g.\ $c(d,\delta)=\delta/2$) from 
\begin{multline}
  \pi_N(|T_N|\ge \delta N)=\pi^1\left(|T_N|\ge \delta N\ | \ (S_n)_{n\in [0,N]} \text{ is self-avoiding } \right)\\
  \le \frac{\pi^1\left(|T_N|\ge \delta N\right)}{\pi^1\left( (S_n)_{n\in [0,N]} \text{ is self-avoiding } \right)} ,
 \end{multline}
 and \eqref{castor}.
 
 \medskip
 
The case of $U_N$ requires a bit more care.
We decompose $|U_N|$ into a sum of two terms. Set $H_n:=\ind_{\{S_{n}=S_{n-3}\}}$

\begin{equation}\label{uone}
|U_N|= \sum_{i=2}^{\lfloor (N-1)/2\rfloor} H_{2i}+\sum_{i=
1}^{\lfloor N/2\rfloor-1} H_{2i+1}=:U^1_N+U^2_N.
\end{equation}
 
 Note that under $\pi_1$, $U^1_N$ and $U^2_N$ are sums of IID Bernoulli random variables of parameter $r(d):=\frac{2(d-1)}{(2d-1)^2}$ (this is because $H_n$ is independent of $(S_k)_{k\le n-2}$).
 Thus using Cramer's Theorem one gets that for any $x<r(d)$ and any $\gep>0$ small enough, for large $N$
 \begin{equation}
\begin{split}\label{gross}
\pi_1\left((2U^1_N/N)<x\right)&\le \exp\left(-N(h_{r(d)}(x)-\gep)/2\right),\\
\pi_1\left((2U^2_N/N)<x \right)&\le \exp\left(-N(h_{r(d)}(x)-\gep)/2\right).
\end{split}
\end{equation}
where $h_{r(d)}$ is defined by \eqref{hp}.

For some fixed small $\delta$, one can get for large $d$ that
\begin{equation}
h_{r(d)}\left(\frac{1-\delta}{2d}\right)>\gd^2/8d.
\end{equation}
Hence choosing $\gep$ small enough and using that  
$$|U_N|/N<x \ \Rightarrow \ \left\{ 2U^1_N/N<x \text{ or } 2U^2_N/N< x \right\}$$
we have for $N$ marge enough
\begin{equation}
\pi_1\left(|U_N|/N<(1-\delta)/(2d)\right)\le \exp\left(-\delta^2 N/(16d)\right). 
\end{equation}
Finally we can conclude (with e.g.\ $c(d,\delta)=\delta^2/(32d)$) by using \eqref{castor} and 
\begin{multline}
  \pi_N(|U_N|/N<(1-\delta)/(2d))\\
  =\pi^1\left(|U_N|/N<(1-\delta)/(2d)\ | \ (S_n)_{n\in [0,N]} \text{ is self-avoiding } \right)\\
  \le \frac{\pi^1\left(|U_N|/N<(1-\delta)/(2d)\right)}{\pi^1\left( (S_n)_{n\in [0,N]} \text{ is self-avoiding } \right)}.
 \end{multline}
 The other bound for $|U_N|$ could be treated similarly but we will see that this is not needed.
 
 \medskip
 
 For the other inequality , let us define two sets whose union equals to $[0, N] \setminus V^N_1$,
 
 \begin{equation}\begin{split}
 W^1_N&:=\{n\in[0,N]\ | \ \exists m< n-1, S_m\sim S_n\},\\
 \bar W^1_N&:=\{n\in[0,N]\ | \ \exists m\in (n+1,N], S_m\sim S_n\}.
 \end{split}\end{equation}
 By invariance under time reversal of the self-avoiding walk, $|W^1_N|$ and $|\bar W^1_N|$ have the same law under $\pi_N$.
Note that  we also have 
  $$|V^1_N|\ge N+1- |W^1_N|-|\bar W^1_N|,$$
 so that it is sufficient to show that 
 \begin{equation}\label{oart}
 \pi_N\left(|W^1_N|\ge \frac{1+\delta}{2d}N\right)\le \exp(-c(d,\delta)N)/2.
 \end{equation}
Recall that

 \begin{equation}\label{astrucar}
 \pi_N(\{|W^1_N|\ge \frac{1+\delta}{2d}N\}) =\frac{\pi^1(|W^1_N|\ge \gep N\ \text{ and } (S_n)_{n\in[0,N]} \text{ is self avoiding })}{\pi^1\left( (S_n)_{n\in [0,N]} \text{ is self-avoiding } \right)}.
 \end{equation}
Hence to our purpose it is sufficient to have a good bound on the numerator.
Let us define $(\tau_n)_{n\ge 0}$ a sequence of stopping time by $\tau_0=0$ and
\begin{equation}
\tau_{n+1}=\min\{k\ge \tau_{n}+1\ | \exists m< k-1,\ S_m\sim S_k\}.
\end{equation}
They are the time at which $S$ comes to a neighborhood of its non-immediate past.
We want to show that

 \begin{equation}\label{croco}
  \pi^1((S_n)_{n\in [1,\tau_N+1]} \text{ is self-avoiding })\le \left(\frac{2(d-1)}{2d-1}\right)^{N}
  \end{equation}

Under $\pi^1$, $\tau_N$ is a stopping time so that by the strong Markov property, a.s.\ on the event "$(S_n)_{n\in[0,\tau_{N}]}$ is self avoiding",
\begin{equation}
 \pi^1(S_{\tau_N+1}\ne S_n,\ \forall n\le \tau_N\ | (S_n)_{n\in[0,\tau_{N}]})\le\left(\frac{2(d-1)}{2d-1}\right).
\end{equation}
This is because at least one out of the $(2d-1)$ available options for $S_{\tau_N+1}$ breaks self-avoidance.
Hence
\begin{multline}
   \pi^1((S_n)_{n\in [1,\tau_N+1]} \text{ is self-avoiding } \ |\  (S_n)_{n\in[0,\tau_{N-1}+1]} \ \text{ is self-avoiding })\\
   =  \pi^1( S_{\tau_N}+1\ne S_n,\ \forall n\le \tau_N\ | \ (S_n)_{n\in[0,\tau_{N-1}+1]} \ \text{ is self-avoiding })\le\left(\frac{2(d-1)}{2d-1}\right).
\end{multline}
Iterating this inequality $N$ times gives \eqref{croco}.

\medskip
 
As a consequence of \eqref{croco} we have
\begin{multline}
 \pi^1\left( \{(S_n)_{n\in [0,N]}  \text{ is self-avoiding }\}\cap \{ |W_1^N|\le \gep N\}\right)\\
 \le 
 \pi^1\left( \{(S_n)_{n\in [0,\tau_{\gep N-1}+1]}  \text{ is self-avoiding }\}\cap \{ \tau_{\gep N-1}\le N-1\}\right)\le \left(\frac{2(d-1)}{2d-1}\right)^{\gep N-1}.
\end{multline}
Then, combining this result with \eqref{astrucar} and \eqref{castor} we obtain
\begin{equation}\label{fopeo}
 \pi_N(|W_N^1|\le \gep N)\le \left(1-\frac{\gep}{2d}+\frac{1}{(2d)^2}+\gep O(d^{-2})+O(d^{-3})\right)^N.
\end{equation}
Choosing $\gep=(1+\delta)/(2d)$ for a fixed $\delta$, we get that for $d$ large enough \eqref{oart} holds.
This gives also the bound for upper-deviation of $|U_N|$ as $U_N\subset W^1_N$.

\end{proof}

The only statement that we need to complete the proof of Proposition \ref{gudespine} is that with large probability $|V^2_N|/N$ is equivalent to one when the asymptotic 
in $d$ is concerned.
This is considerably more complicated than for $|V_N^1|$ because in this case $\pi_1$ is not a fine enough approximation of the measure to get a conclusion.

\begin{proposition}
There exists a constant $C$ such that for all $d$ large enough there exists $c(d)$ such that for all $N\ge N_0(d)$ 
\begin{equation}
  \pi_N\left(|V^2_N|\le \left(1-\frac{C}{d}\right)N\right)\le e^{-c(d)N},
\end{equation}

\end{proposition}

\begin{proof}

For the same reasons as in the proof of Proposition \ref{vunetu} (see \eqref{oart} and the few lines above it), we define
\begin{equation}
W^2_N:=\{n\in [0,N]\ | \ \exists m< n-2, S_m\sim S_n\},
\end{equation}
and focus on proving that 

\begin{equation}
  \pi_N\left(|W^2_N|\ge \frac{C}{d}N\right)\le e^{-c(d)N},
\end{equation}
(for a different constant $C$).

\medskip

We first want to prove a result concerning the measure $\pi^2$.
The reason why we do our proof for $\pi^2$ and not for $\pi_N^2$ is that having a Markovian walks is a useful tool for the proof.
The plan is then to transform it into a result for $\pi^2_N$, and finally to use \eqref{castor} to conclude.
To this purpose we define $\tau_N^2$  by $\tau^2_0=0$, and

\begin{equation}\label{deftau2}
\tau^{2}_{N+1}:= \min \left\{ n\ge \tau^2_N +2 \  |  \exists m< n-2, |S_n-S_m|=2; \{n,n-1\} \cap W^1_N=\emptyset \right\}.
\end{equation}
Approximately, $\tau^{2}$ is the sequence of time at which $S$ comes at distance two of its non-immediate past 
(it is not exactly right because of the extra conditions $n\ge \tau^2_N +2$ and $\{n,n-1\} \cap W^1_N=\emptyset$).
We want to prove that 

\begin{equation}\label{degueu}
\pi^2( (S_n)_{n\in[0,\tau^2_N+2]} \text{ is self-avoiding })\le \left(1-\frac{1}{(2d-1)^2}\right)^{N}.
\end{equation}

The idea is again to use the Markov property similarly as what we did to prove \eqref{croco}. Given a self-avoiding trajectory $(S_n)_{n\in[0,\tau^2_N]}$, one can choose some $m<\tau^2_N-2$ such that
$|S_{\tau^2_N}-S_m|=2$ (there might be several choices but we fix the value of $m$ for what follows e.g.\ we take the largest possible $m$).

There is at least one way to reach $S_m$ in two steps from $S_{\tau_N}$. We have to check that such a combination of two step is authorized for $\pi^2$, and in fact because 
$S_m\notin \{S_{\tau^2_N}, S_{\tau^2_{N}-2}\}$, it is sufficient to check the first step is not a backtrack and does not  form a loop of length four.
Because $\tau_N-1\notin W^1_N$, a step going to a neighbor of $S_m$ cannot be a backtrack, and as $\tau_N\notin W^1_N$ 
it cannot create a loop of length four in the first step anyway.

\medskip

As there are at most $(2d-1)$ possibilities for each step, we deduce that the possibility of reaching $S_m$ in two is at least 
$(2d-1)^{-2}$.
Averaging over all the possibilities for  $(S_n)_{n\in[0,\tau^2_N]}$, we deduce that

\begin{multline}\label{degeu}
\pi^2\left( (S_n)_{n\in[0,\tau^2_N+2]} \text{ is self-avoiding }\ | \ (S_n)_{n\in[0,\tau^2_{N-1}+2]} \text{ is self avoiding }\right)
\\ \le \pi^2\left( (S_n)_{n\in[0,\tau^2_N+2]} \text{ is self-avoiding }\ | \ (S_n)_{n\in[0,\tau^2_{N}]} \text{ is self avoiding }\right)\le 1-\frac{1}{(2d-1)^2}.
\end{multline}

 We obtain \eqref{degueu} by iteration of \eqref{degeu}.

\medskip

Now, because of the conditions  $\{n,n-1\} \cap W^1_N=\emptyset$ and $n\ge \tau^2_N +2$ in \eqref{deftau2}
the sequences $(\tau^2_N)$ is not the generalized inverse of $(|W^2_N|)$.
However \eqref{degueu} combined with additional results can still be useful to get a bound on the tail distribution of $|W^2_N|$.
Set 
\begin{equation}
\cT_N:=\max \{ n \ |\ \tau^2_n\le N \}.
\end{equation}
We have
\begin{equation}
|W^2_{N}|\le  2 (\cT_N+|W^1_N|).
\end{equation}
This is because if $n\in W^2_N$, either it is equal to one of the $(\tau^2_n)_{1\le n\le T_N}$ or $(\tau^2_{n-1})_{2\le n\le T_N+1}$
or it belongs to $( W^1_N+1)\cup W^1_N$.
Hence
\begin{equation}\label{astare}
\pi_N( |W^2_{N}|\ge 4\gep N)\le \pi_N(|\tau^2_{\gep N}|\le N)+\pi_N( |W^1_{N}|\ge \gep N).
\end{equation}
We want to get a bound on the right hand side for $\gep=5C/d$, and we know that the second term is small from \eqref{oart}.
We will use \eqref{degueu} to bound the first term.
Notice that for $N$ large enough
\begin{multline}\label{jugg}
\pi^2( \tau^2_{\gep N}\le N\ , \ (S_{n})_{n\in [0,N]} \text{ is self-avoiding })\\ \le
\pi^2((S_{n})_{n\in [0,\tau_{\gep N-1}+2]} \text{ is self-avoiding })
\\
\le
\left(1-\frac{1}{(2d-1)^2}\right)^{\gep N-1}\le \exp(-\gep N/(2d)^2).
\end{multline}

Then one wants to compare $\pi^2( \tau^2_{\gep N}\le N\ , \ (S_{n})_{n\in [0,N]} \text{ is self-avoiding })$ to the probability of the same event under $\pi^2_N$. Using the Cauchy-Schwarz inequality one gets 

\begin{multline}\label{gini}
\pi^2_N( \tau^2_{\gep N}\le N\ , \ (S_{n})_{n\in [0,N]} \text{ is self-avoiding })^2
\\ \le \pi^2\left(\left(\frac{\partial \pi^2_N}{\partial \pi^2}\right)^2\right) \pi^2( \tau^2_{\gep N}\le N\ , \ (S_{n})_{n\in [0,N]} \text{ is self-avoiding }),
\end{multline}
where $\pi(f)$ denotes the expectation of $f$ under the probability measure $\pi$.
We know that the second term of the above equation decays exponentially fast. We focus on showing that the first term does not increase so fast as to counterbalance it.
From \eqref{radpi} we have
\begin{equation}\label{cosmogol}
\pi^2\left(\left(\frac{\partial \pi^2_N}{\partial \pi^2}\right)^2\right) =\pi^2\left( \left(\frac{2d-2}{2d-1}\right)^{2|U_{N-1}|}\right) \left(\pi^2\left(\frac{2d-2}{2d-1}\right)^{|U_{N-1}|}\right)^{-2}.
\end{equation}
The second factor is not too hard to bound from below, using Jensen's inequality

\begin{equation}\label{micranots}
\pi^2\left(\left(\frac{2d-2}{2d-1}\right)^{|U_{N-1}|}\right)\ge\left( \frac{2d-2}{2d-1}\right)^{\pi^2(|U_{N-1}|)},
\end{equation}
and the expectation $\pi^2(|U_{N-1}|)$ is not difficult to approximate:
for $n\in [3,N-1]$ we have 
$\pi_2(n\in U_{N-1})\le \frac{1}{2d-2}$, because there are at least $2d-2$ choices for $S_n$ (at most $2$ are forbidden)
and at most one for which $n\in U_N$ is satisfied.
Hence 
$$\pi^2(|U_{N-1}|) \le \frac{N-3}{2d-2}.$$

The first one is is a bit more delicate to control.

\begin{lemma}
For all $x\le 1$ one has 
\begin{equation}
\pi^2(x^{|U_N|})\le \left(1+(1-x^2) \frac{2d-3}{(2d-1)^2}\right)^{N/2-2}.
\end{equation}
\end{lemma}

\begin{proof}
We use the decomposition \eqref{uone} and Cauchy-Schwarz to get
\begin{equation}
\pi^2\left(x^{|U_N|}\right)\le \sqrt{\pi^2\left(x^{2U^1_N}\right)\pi^2\left(x^{2U^2_N}\right)}.
\end{equation}
Then we remark that almost surely
\begin{equation}
\pi^2(S_{N}=S_{N-3} | (S_n)_{n\le N-2})\ge \frac{2d-3}{(2d-1)^2}
\end{equation}
Indeed there are $(2d-2)$ possibilities for making a $U$-turn and at most one of them is made unavailable by the condition that $4$ loops are not allowed.
The probability of a fixed combination of two available steps is at least $(2d-1)^{-2}$, and
hence, (recall $H_N:=\ind_{S_N\sim S_{N-3}}$)
\begin{equation}
\pi^2(H_N=1\ | \ (H_n)_{n\le N-2})\ge \frac{2d-3}{(2d-1)^2}.
\end{equation}
By a trivial induction (recall the definition \eqref{uone})
\begin{equation}
\pi^2\left(x^{2U^i_N}\right)\le \left(1- (1-x^2) \frac{2d-3}{(2d-1)^2}\right)^{N/2-2}.
\end{equation}

\end{proof}

The previous lemma together with equations \eqref{cosmogol} and \eqref{micranots}
\begin{equation}
 \pi^2\left(\left(\frac{\partial \pi^2_N}{\partial \pi^2}\right)^2\right)\le \left(\frac{2d-1}{2d-2}\right)^{\frac{N-3}{d-1}}
 \left(1+\left(1-\left(\frac{2d-2}{2d-1}\right)^4\right)\frac{2d-3}{(2d-1)^2}\right)^{N/2-2}.
\end{equation}
Performing Taylor expansions in $(2d)^{-1}$,
we see that
\begin{equation}
 \left(\frac{2d-1}{2d-2}\right)^{\frac{N-3}{d-1}}=\left(1+\frac{1}{2d^2}+O(d^{-3})\right)^N,
\end{equation}
and that 
\begin{multline}
 \left(1-\left(1-\left(\frac{2d-2}{2d-1}\right)^4\right)\frac{2d-3}{(2d-1)^2}\right)^{N/2-2}
\\ =\left(1+\frac{2}{d}\frac{1}{2d}(1+O(d^{-1})\right)^{N/2-2}=\left(1+\frac{1}{2d^2}+O(d^{-3})\right)^{N},
 \end{multline}
and we can conclude that there exists $C$ (which we might choose as large as we wish) such that for all $d$ large enough, and all $N\ge N_0(d)$
\begin{equation}
 \pi^2\left(\left(\frac{\partial \pi^2_N}{\partial \pi^2}\right)^2\right)\le \exp(CN/d^3).
\end{equation}
Recall that from \eqref{jugg}, fixing $\gep:=20C/d$ we have
\begin{equation}
\pi^2( \tau^2_{\gep N}\le N\ , \ (S_{n})_{n\in [0,N]} \text{ is self-avoiding })\le \exp(-\gep N/(2d)^2)\le  \exp(-5C/d^3).
\end{equation}
This allows us to conclude from \eqref{gini} that

\begin{equation}
\pi^2_N( \tau^2_{\gep N}\le N\ , \ (S_{n})_{n\in [0,N]} \text{ is self-avoiding })\le \exp(-2NC/d^3).
\end{equation}
Finally we combin \eqref{castor} and
\begin{equation}
\pi_N( \tau^2_{\gep N}\le N)=\frac{\pi^2_N( \tau^2_{\gep N}\le N\ , \ (S_{n})_{n\in [0,N]} \text{ is self-avoiding })}{\pi^2_N( (S_{n})_{n\in [0,N]} \text{ is self-avoiding })},
\end{equation}
and provided $C$ has been chosen large enough, we obtain that

\begin{equation}
\pi_N( \tau^2_{\gep N}\le N)\le  \exp(-NC/d^3).
\end{equation}

Going back to \eqref{astare} this implies the desired result once combined with \eqref{fopeo}.

\end{proof}

\end{document}